\documentclass[12pt]{article}%
\usepackage{amsmath}
\usepackage{amsfonts}
\usepackage{amssymb}
\usepackage{graphicx}%
\providecommand{\U}[1]{\protect\rule{.1in}{.1in}}
\newtheorem{theorem}{Theorem}

\newtheorem{corollary}[theorem]{Corollary}

\newtheorem{example}[theorem]{Example}

\newtheorem{lemma}[theorem]{Lemma}

\newtheorem{proposition}[theorem]{Proposition}

\newenvironment{proof}[1][Proof]{\noindent\textbf{#1.} }{\ \rule{0.5em}{0.5em}}
\begin{document}

\title{Order denseness in free Banach lattices}
\author{Youssef Azouzi and Wassim Dhifaoui\\{\small Research Laboratory of Algebra, Topology, Arithmetic, and Order}\\{\small Department of Mathematics}\\{\small Faculty of Mathematical, Physical and Natural Sciences of Tunis}\\{\small Tunis-El Manar University, 2092-El Manar, Tunisia}}
\date{}
\maketitle

\begin{abstract}
We prove a fundamental property: the free vector lattice $\operatorname*{FVL}%
\left[  E\right]  $ over a Banach space $E$ is order dense in the free
$p$-convex Banach lattice $\operatorname*{FBL}\nolimits^{(p)}\left[  E\right]
,$ $1\leq p\leq\infty,$ if and only if $E$ is finite-dimensional. In a recent
work, Oikhberg, Tradacete, Taylor, and Troitsky \cite[Theorem 2.8]{L-925}
claimed that order denseness holds for all Banach spaces. We point out a gap
in their proof, and consequently, any conclusions relying on this claim
require reexamination---a task we also undertake in this present paper. A key
tool in our approach, which also leads to a partial answer to an open question
recently posed by these authors.

\end{abstract}


\section{Introduction}

Over the past decade, a new and promising area in functional analysis has
emerged, linking Banach spaces with Banach lattices. This topic has attracted
significant attention from researchers. Two fundamental papers mark the
beginning of this line of study. The first paper \cite{L-332}, by de Pagter
and Wickstead, introduced the concept of free Banach lattices generated by a
set, denoted by $\operatorname*{FBL}\left(  A\right)  ,$ effectively
initiating the field. Three years later, a second major contribution
\cite{L-403} came from Avil\'{e}s, Rodr\'{\i}guez, and Tradacete, who
generalized the theory by introducing the notion of free Banach lattice
generated by a Banach space, denoted by $\operatorname*{FBL}\left[  E\right]
$. Notably, the free Banach lattice generated by a set $A$ corresponds to the
free Banach lattice generated by the Banach space $\ell_{1}\left(  A\right)
$, that is $\operatorname*{FBL}\left(  A\right)  =\operatorname*{FBL}\left[
\ell_{1}\left(  A\right)  \right]  .$ Since then, research activity in this
area has accelerated rapidly, with over thirty papers published within just a
few years. This growing body of work highlights the vitality and potential of
the subject. Our contribution is motivated by and builds upon the pioneering
paper \cite{L-925}, which appeared in the Journal of the European Mathematical
Society one year ago.

By definition, the free vector lattice $\operatorname*{FVL}\left[  E\right]  $
of a Banach space $E$ is norm dense in the free Banach lattice
$\operatorname*{FBL}\nolimits^{(p)}\left[  E\right]  $ generated by the same
space. The question of order denseness, however, is more subtle and was
discussed in \cite{L-925}, where it was shown that $\operatorname*{FVL}\left[
E\right]  $ is order dense in $\operatorname*{FBL}\left[  E\right]  $ for
every Banach space $E.$ But there was a gap in the proof and it turns out that
order denseness holds only when $E$ is finite-dimensional. The aim of this
note is to give a proof of the latter result and provide some consequences. In
particular, we will revisit some previously established results using order
density and show that, with one exception, they remain valid. One of our key
tools to attack the problem of denseness is the use of bounded weak* topology.
This topology is very suitable when handling free Banach lattices because,
contrary to weak* topology, elements in $\operatorname*{FBL}\nolimits^{(p)}%
\left[  E\right]  $ act continuously with respect to bounded weak topology on
$E^{\ast}.$ The paper is organized as follows. Section 2 introduces the
necessary notation and preliminaries. We begin by fixing some notation and
recalling basic facts about free Banach lattices that will be used throughout
the paper. We then give a brief overview of the bounded weak* topology, a key
tool in our work. In section 3 we present our results, starting with the main
theorem, which states that the free vector lattice $\operatorname*{FVL}\left[
E\right]  $ is order dense in the free Banach lattice $\operatorname*{FBL}%
\nolimits^{(p)}\left[  E\right]  $ if and only $E$ is finite-dimensional. We
show then an operator $T:F\longrightarrow E$ is injective if and only if the
natural lattice homomorphism $\overline{T}:\operatorname*{FBL}%
\nolimits^{(\infty)}\left[  F\right]  \longrightarrow\operatorname*{FBL}%
\nolimits^{(\infty)}\left[  E\right]  $ associated to $T$ is injective. Next,
we provide a partial answer to a question raised in \cite{L-925} by Oikhberg,
Taylor, Tradacete and Troitsky showing that if $T$ is with closed range then
$\overline{T}:\operatorname*{FBL}\nolimits^{(\infty)}\left[  F\right]
\longrightarrow\operatorname*{FBL}\nolimits^{(\infty)}\left[  E\right]  $ is
order continuous if and only if $T$ is injective. In fact, we prove that order
continuity alone of $\overline{T}$ implies injectivity of $T,$ without
assuming the closedness of the range. For notations and terminology on vector
lattices and Banach lattices not recalled here, the reader is referred to
\cite{b-240}.

\section{Preliminaries}

\textbf{Free Banach lattices. }We begin with a brief review and some essential
terminology related to free Banach lattices that will be used throughout the
paper. For foundational background, the reader is referred to the seminal
papers \cite{L-332,L-403}. All additional material necessary for reading this
work can be found in the memoire \cite{L-925}. Given a Banach space $E,$ the
free $p$-convex Banach lattice generated by $E$ is a Banach lattice $L$
satisfying some universal property. It was defined for $p=1$ in \cite{L-403}
and then extended in \cite{L-742} for the other values of $p.$ A complete
description of these spaces and their construction can be found in
\cite{L-925}. The space $\operatorname*{FBL}\nolimits^{(p)}\left[  E\right]  $
is a $p$-convex  Banach lattice $\operatorname*{FBL}\nolimits^{(p)}\left[
E\right]  ($\footnote{Note that a $1$-convex Banch lattice is just a Banach
lattice, while an $\infty$-convex Banach lattice is nothing other an
AM-space.}) together with a linear isometry $\delta_{E}:E\longrightarrow
\operatorname*{FBL}\nolimits^{(p)}\left[  E\right]  $ such that for every
$p$-convex Banach lattice $X$ and every bounded operator $T:E\longrightarrow
X$ there exists a unique lattice homomorphism $\widehat{T}:\operatorname*{FBL}%
\nolimits^{(p)}\left[  E\right]  \longrightarrow X$ such that $T=\widehat
{T}\circ\delta$ and $\left\Vert \widehat{T}\right\Vert \leq M_{p}\left(
X\right)  \left\Vert T\right\Vert ,$ where $M_{p}\left(  X\right)  $ is the
$p$-convexity constant of $X.$ For a concrete construction of
$\operatorname*{FBL}\nolimits^{(p)}\left[  E\right]  $ we refer the reader to
\cite[Theorem 2.1 and Proposition 2.2]{L-925}. Having an explicit description
of these spaces can be useful in proving certain results, though in this work
we rely only a few of their basic properties. We simply emphasize that
$\operatorname*{FBL}\nolimits^{(p)}\left[  E\right]  $ is a sublattice of
$H[E],$ the linear subspace of $\mathbb{R}^{E^{\ast}}$ consisting of all
positively homogeneous functions. It is in fact the closed sublattice
generated by $\delta_{E}\left(  E\right)  $ in the Banach space $H^{p}\left[
E\right]  =\left\{  f\in H\left[  E\right]  :\left\Vert f\right\Vert
_{\operatorname*{FBL}\nolimits^{(p)}\left[  E\right]  }<\infty\right\}  ,$
where the norm is defined as follow.%
\[
\left\Vert f\right\Vert _{\operatorname*{FBL}\nolimits^{(p)}\left[  E\right]
}=\sup\left(
{\textstyle\sum\nolimits_{k}}
\left\vert f\left(  x_{k}^{\ast}\right)  \right\vert ^{p}\right)  ^{1/p}.
\]
Here the supremum is taken over all finite sequences $\left(  x_{k}^{\ast
}\right)  $ in $E^{\ast}$ satisfying
\[
\sum\limits_{k}\left\vert x_{k}^{\ast}\left(  x\right)  \right\vert ^{p}%
\leq1\text{ for all }x\in B_{E}.
\]

The map $\delta_{E}:E\longrightarrow H\left[  E\right]  $ enabling to identify
$E$ with a sublattice of $\operatorname*{FBL}\nolimits^{(p)}\left[  E\right]
$ is given by $\delta_{E}\left(  x\right)  \left(  x^{\ast}\right)  =x^{\ast
}\left(  x\right)  $ for $x\in E$ and $x^{\ast}\in E^{\ast}.$

The case $p=\infty$ is more interesting for us. In this case $\left\Vert
f\right\Vert _{\operatorname*{FBL}\nolimits^{(\infty)}\left[  E\right]  }%
=\sup\limits_{x^{\ast}\in B_{E^{\ast}}}\left\vert f\left(  x^{\ast}\right)
\right\vert $ and $\operatorname*{FBL}\nolimits^{(\infty)}\left[  E\right]  $
is again the close sublattice of $H\left[  E\right]  $ generated by
$\delta\left(  E\right)  .$ It was shown in \cite[Proposition 2.2]{L-925} that
$\operatorname*{FBL}\nolimits^{(\infty)}\left[  E\right]  $ is in fact the
space of all $f$ in $H\left[  E\right]  $ which are weak* continuous on the
unit ball $B_{E^{\ast}}.$ Initially, the spaces $H\left[  E\right]  $ and
$H^{p}\left[  E\right]  $ were introduced primarily to study
$\operatorname*{FBL}\nolimits^{(p)}\left[  E\right]  ,$ which was the main
focus of authors. Subsequently, several other sublattices of $H\left[
E\right]  $ have been investigated, notably in \cite{L-968}. Here, we restrict
ourselves to mentioning the following ones, as we will refer to them in a
specific result. The ideal of $H\left[  E\right]  $ generated by $\{\delta
_{x}:x\in E\}$ will be denoted by $I\left[  E\right]  ,$ while the
$H_{w^{\ast}}[E]$ is consisting of those functions $f\in H[E]$ that are weak*
continuous on the ball $B_{E^{\ast}}.$ We also define $I_{w^{\ast}%
}[E]=I[E]\cap H_{w^{\ast}}\left[  E\right]  $ and $H_{w^{\ast}}^{p}%
[E]=H^{p}\left[  E\right]  \cap H_{w^{\ast}}\left[  E\right]  .$

On of the central motivations for studying free Banach lattices lies in the
rich interplay between Banach spaces and Banach lattices. One of such
connections---extensively explored in the monograph \cite{L-925} and central
to our approach -- is the relationship between the operator $T$ and its
lattice counterpart $\overline{T}.$ Let us say a few words about that. Given
two Banach spaces $E$ and $F.$ Every bounded operator $T:F\longrightarrow E$
gives rise in a natural way to a lattice homomorphism $\overline
{T}:\operatorname*{FBL}\nolimits^{(p)}\left[  F\right]  \longrightarrow
\operatorname*{FBL}\nolimits^{(p)}\left[  E\right]  .$ First we compose with
$\delta_{E}$ to get a bounded operator $\delta_{E}\circ T:F\longrightarrow
\operatorname*{FBL}\nolimits^{(p)}\left[  E\right]  ,$ then we use the
universal property to extend this operator to $\operatorname*{FBL}%
\nolimits^{(p)}\left[  F\right]  .$ That $\overline{T}=\widehat{\delta
_{E}\circ T}.$ In \cite{L-925} the authors showed that $\overline{T}f=f\circ
T^{\ast}$ for all $f\in\operatorname*{FBL}\nolimits^{(p)}\left[  F\right]  .$
One of the main ideas considered in \cite{L-925} is to study the relationship
between properties of $T$ and corresponding properties of $\overline{T}.$ For
instance it was shown in that $T$ is injective (respectively, surjective, with
dense range) if and only if $\overline{T}$ is so. We will give a new proof for
the first point as its original proof is invalid as explained in Section 3.

\textbf{Bounded weak* topology. }Since this topology will be used multiple
times, we begin with a brief overview. Let $E$ be a normed space and let
$(x_{n})$ be a sequence in $E$ that converges to $0.$ For $x^{\ast}\in
X^{\ast}$ we let $B(x^{\ast},(x_{n}))$ the set of all elements $y^{\ast}$ in
$E^{\ast}$ satisfying $\left\vert \left(  y^{\ast}-x^{\ast}\right)  \left(
x_{n}\right)  \right\vert <1$ for each all $n.$ The collection of sets
$B(x^{\ast},(x_{n}))$ forms a basis of the topology known as the bounded
weak*\ topology of $E^{\ast}.$ A net $(x_{\alpha}^{\ast})_{\alpha\in A}$
converges to $x^{\ast}$ with respect to this topology, written $x_{\alpha
}^{\ast}\overset{bw^{\ast}}{\longrightarrow}x^{\ast}$, if $\left(  x_{\alpha
}^{\ast}\right)  _{\alpha\in A}$ converges uniformly to $x^{\ast}$ on each
norm compact subset of $E.$ In \cite{a-2643} Dieudonn\'{e} showed that this is
equivalent to the fact that the above convergence holds on each sequence
converging to $0.$ For further knowledge about this topology we refer the
reader to the paper \cite{a-2643} and the books \cite{b-2554,b-1901}.

It is worth noting that bounded weak* converging sequences are norm bounded,
but this property does not extend to nets. Indeed, a bounded weak* convergent
net is not necessarily eventually bounded, as the following example will show.

\begin{example}
Let $E$ be a nonseparable Banach space. For each compact $K$ in $E$ and each
integer $n$ in $\mathbb{N}$ we can find $x^{\ast}\in E^{\ast}$ such that
$x^{\ast}\left(  K\right)  =0$ and $\left\Vert x^{\ast}\right\Vert =n;$ this
linear functional will be denoted by $x_{\left(  K,n\right)  }^{\ast}.$ It is
obvious that the net $\left(  x_{\left(  K,n\right)  }^{\ast}\right)  $ is not
norm bounded and eventually vanishes on each compact subset of $E,$ hence
$x_{\left(  K,n\right)  }^{\ast}\overset{bw^{\ast}}{\longrightarrow}0.$ Here
the set $\left(  K,n\right)  \in\mathcal{K}\times\mathbb{N},$ where
$\mathcal{K}$ denotes the set of all compact subsets in $E,$ is ordered in a
natural way: $\left(  K,n\right)  \leq\left(  K^{\prime},n\right)
\Leftrightarrow K\subseteq K^{\prime}$ and $n\leq n^{\prime}.$
\end{example}

Recall form \cite[Corollary 2.5.7]{b-2554} that a subset $A$ of a dual space
$X^{\ast}$ is bounded weak* closed if and only if $A\cap nB_{X^{\ast}}$ is
weak* closed for every $n\in\mathbb{N}.$ Having this in mind,
Banach-Dieudonn\'{e} Theorem (see \cite[Theorem 3.92]{b-1417}) in equivalent
to the following, which may be viewed as an analogue of Mazur's Theorem.

\begin{theorem}
\label{BD}Let $X$ be a Banach space \textit{and let} $A$ \textit{be a convex
set in} $X^{\ast}$. Then $A$ is bounded weak* closed if and only if $A$ is
weak* closed.\textit{ In particular we have }$\overline{A}^{bw^{\ast}%
}=\overline{A}^{w^{\ast}}.$
\end{theorem}

This result is also referred to as Krein-\u{S}mulian Theorem in
\cite[Corollary 8.47]{b-1901}.

It was noted in \cite{L-925} that elements in the free Banach lattice
$\operatorname*{FBL}\nolimits^{(p)}\left[  E\right]  $ are bounded weak*
continuous. For the sake of completeness, we provide a proof of this result,
as it plays a crucial role in our work.

\begin{lemma}
Let $E$ be a Banach space and $p\in\left[  1,\infty\right]  .$ Then every
$f\in\operatorname*{FBL}\nolimits^{(p)}\left[  E\right]  $ is bounded weak*
continuous on $E^{\ast}$.
\end{lemma}

\begin{proof}
Let $f\in\operatorname*{FBL}\nolimits^{(p)}\left[  E\right]  $ and let $A$ be
a closed set in $\mathbb{R}.$ We want to show that $f^{-1}\left(  A\right)  $
is bounded weak* closed. To this end let $\left(  x_{\alpha}^{\ast}\right)  $
be a bounded net in $f^{-1}\left(  A\right)  $ such that $x_{\alpha}^{\ast
}\overset{bw^{\ast}}{\longrightarrow}x^{\ast}.$ Then in particular,
$x_{\alpha}^{\ast}\overset{w^{\ast}}{\longrightarrow}x^{\ast}.$ But as $f$ is
weak* continuous on bounded subsets of $E^{\ast}$ we get $f\left(  x_{\alpha
}^{\ast}\right)  \longrightarrow f\left(  x^{\ast}\right)  \in A$ as $A$ is
closed in $\mathbb{R}$. According to \cite[Lemma 2.7.5]{b-2554} this proves
that $f$ is bounded weak* continuous as it was claimed.
\end{proof}

\section{Order denseness}

Our main result is this section is Theorem \ref{Main} which asserts that a
Banach lattice $E$ is finite-dimensional if and only if $\operatorname*{FVL}%
\left[  E\right]  $ is order dense in $\operatorname*{FBL}\nolimits^{(p)}%
\left[  E\right]  $ for any $p\in\left[  1,\infty\right]  .$ Before proving
this theorem we require a lemma. This lemma was originally stated in
\cite{L-925}, and we provide here an easy and short proof of it. The version
we present is slightly more general, making it more convenient to employ in
subsequent arguments.

\begin{lemma}
\textit{\label{F}}Let $E$ \textit{be a finite-dimensional Banach space, }$C$
an\textit{\ open cone in} $E^{\ast},$ $0\neq x_{0}^{\ast}\in C$, $B$ a bounded
set in $E^{\ast},$ \textit{and }$\varepsilon>0$. \textit{Then there exists}
$g=g\left(  E,C,x_{0}^{\ast},B,\varepsilon\right)  \in\operatorname*{FVL}%
[E]_{+}$ \textit{such} \textit{that} $g(x_{0}^{\ast})>0,\ g\leq\varepsilon$
\textit{on} $B\cap C$, \textit{and} $g$ \textit{vanishes outside} $C$.
\end{lemma}

Before giving the detailed proof let us explain why we can reduce ourselves to
the case $E=\ell_{1}^{n}.$ Assume that the lemma is proved for a
finite-dimensional space $F$ and assume that $E$ is isomorphic to $F.$ Let
$T:E\longrightarrow F$ be such isomorphism and denote by $S$ its inverse. We
know that $T$ can be extended to a unique lattice isomorphism $\overline
{T}:\operatorname*{FBL}\left[  E\right]  \longrightarrow\operatorname*{FBL}%
\left[  F\right]  $ with inverse $\overline{S}.$ If $g_{1}=g\left(  S^{\ast
}\left(  C\right)  ,S^{\ast}\left(  B\right)  ,S^{\ast}x_{0}^{\ast
},\varepsilon\right)  $ for $F$ then $\overline{S}g_{1}$ works for $E.$

\begin{proof}
We can assume without loss of generality that $E=\ell_{1}^{n}$ and, by
homogeneity, we may assume that $x_{0}^{\ast}\in S:=S_{\ell_{1}^{n}}$ and
$B\subseteq B_{\ell_{\infty}^{n}}.$ Let us identify $\operatorname*{FBL}%
\left[  E\right]  $ with $C\left(  S\right)  .$ Let $\delta\in\left(
0,\varepsilon/2\right)  $ such that $K=B\left(  x_{0}^{\ast},3\delta\right)
\cap S\subseteq C$ and let $x_{1}^{\ast}\in S\cap U$ with $\left\Vert
x_{0}^{\ast}-x_{1}^{\ast}\right\Vert _{\infty}=\delta.$ Let $\left(
e_{1},...,e_{n}\right)  $ be the canonical base of $\ell_{1}^{n}.$ As
$\mathbf{1}_{S}=\bigvee\limits_{k=1}^{n}\left\vert \delta_{e_{k}}\right\vert
\in\operatorname*{FVL}\left[  E\right]  $, It is clear that $g:x^{\ast
}\longmapsto\left(  2\delta-\left\Vert x^{\ast}-x_{1}^{\ast}\right\Vert
\right)  ^{+}$ is suitable.
\end{proof}

We come now to the main result of this paper.

\begin{theorem}
\label{Main}Let $E$ be a Banach space and $p\in\left[  1,\infty\right]  .$
Then $\operatorname*{FVL}\left[  E\right]  $ is order dense in
$\operatorname*{FBL}\nolimits^{(p)}\left[  E\right]  $ if and only if $E$ is finite-dimensional.
\end{theorem}

\begin{proof}
\textbf{Finite-dimensional case.} We identify $\operatorname*{FBL}%
\nolimits^{\left(  p\right)  }\left[  E\right]  $ with the space of continuous
functions on the sphere $S_{E^{\ast}}$ and we fix an element $f\in
\operatorname*{FBL}\nolimits^{\left(  p\right)  }\left[  E\right]  $ with
$f>0.$ Then for some open set $U$ of $S$ and some $\varepsilon>0$ we have
$f\left(  x\right)  \geq\varepsilon$ for all $x\in U.$ Fix an element
$x_{0}^{\ast}$ in $U$ and consider the cone $C$ generated by $U.$ The function
$g$ given by Lemma \ref{F} associated to $E,C,S,\varepsilon$ and $x_{0}^{\ast
}$ satisfies : $g<\varepsilon\leq f$ on $U.$ Thus as $f$ and $g$ are
positively homogeneous $g\leq f$ on $C=cone\left(  U\right)  .$ Moreover as
$g$ vanishes outside $C$ this inequality remains valid on $E^{\ast},$ which
shows the order denseness.

\textbf{Infinite-dimensional case.} We will prove that $\operatorname*{FVL}%
\left[  E\right]  $ is never order dense in $\operatorname*{FBL}%
\nolimits^{(p)}\left[  E\right]  .$ In this case we choose a sequence $\left(
x_{n}\right)  _{n\geq1}$ of linearly independent vectors in $E$ satisfying
$\left\Vert x_{n}\right\Vert \leq2^{-n}$ and $u=%
{\textstyle\sum\limits_{n=1}^{\infty}}
x_{n}\neq0.$ Moreover let $u^{\ast}$ be a linear functional in $E^{\ast}$
satisfying $u^{\ast}\left(  u\right)  =2.$ Thus $u^{\ast}\left(  u_{n}\right)
\geq1$ for $n$ large enough, where $u_{n}=%
{\textstyle\sum\limits_{k=1}^{n}}
x_{k}.$ By omitting a finite number of terms of the sequence $\left(
u_{n}\right)  ,$ we may assume now that the sequence $\left(  u_{n}\right)  $
satisfies the following conditions:

(i) $u_{1},u_{2},...$ are linearly independent;

(ii) $\left\Vert u_{n+1}-u_{n}\right\Vert \leq2^{-n};$

(iii) $u^{\ast}\left(  u_{n}\right)  \geq1$ for all $n.$

Now let $f_{n}=\bigwedge\limits_{k=1}^{n}\left\vert u_{k}\right\vert
\in\operatorname*{FVL}\left[  E\right]  .$ Using Birkhoff inequality we
observe that%
\[
0\leq f_{n}-f_{n+1}\leq\left\vert u_{n}\right\vert -\left\vert u_{n}%
\right\vert \wedge\left\vert u_{n+1}\right\vert \leq\left\vert u_{n}%
-u_{n+1}\right\vert .
\]
Thus%
\[
\left\Vert f_{n}-f_{n+1}\right\Vert _{\operatorname*{FBL}\nolimits^{\left(
p\right)  }\left[  E\right]  }\leq\left\Vert u_{n}-u_{n+1}\right\Vert
_{\operatorname*{FBL}\nolimits^{\left(  p\right)  }\left[  E\right]
}=\left\Vert u_{n}-u_{n+1}\right\Vert _{E}\leq2^{-n}.
\]
Hence $f:=\lim f_{n}$ is well defined in $\operatorname*{FBL}%
\nolimits^{\left(  p\right)  }\left[  E\right]  $ and as $\left(
f_{n}\right)  $ is decreasing we have also $f_{n}\downarrow f.$ As%
\[
f_{n}\left(  u^{\ast}\right)  =\bigwedge\limits_{k=1}^{n}\left\vert
u_{k}\right\vert \left(  u^{\ast}\right)  =\bigwedge\limits_{k=1}%
^{n}\left\vert u_{k}\left(  u^{\ast}\right)  \right\vert =\bigwedge
\limits_{k=1}^{n}\left\vert u^{\ast}\left(  u_{k}\right)  \right\vert \geq1
\]
for all $n$ and $f_{n}\longrightarrow f$ in $\operatorname*{FBL}%
\nolimits^{(p)}\left[  E\right]  $ we have $f\left(  u^{\ast}\right)  \geq1.$
In particular $f>0.$

Assume now that $g\in\left[  0,f\right]  \cap\operatorname*{FVL}\left[
E\right]  $ and write $g$ as a lattice linear expression of some elements in
$E.$ Say $g=H\left(  z_{1},...,z_{n}\right)  $ for some $z_{1},...,z_{n}$ in
$E.$ Then if we consider $F=\mathrm{span}\left(  z_{1},...,z_{n}\right)  $ and
$y^{\ast}\in E^{\ast}$ we have $g\left(  y^{\ast}\right)  =g\left(  z^{\ast
}\right)  $ for every $z^{\ast}\in E^{\ast}$ that agrees with $y^{\ast}$ on
$F.$ Now by condition (i) above we can choose an element $u_{k}$ not belonging
to $F$ and by Hahn-Banach theorem we can find $z^{\ast}\in E^{\ast}$ such that
$z^{\ast}$ agrees with $y^{\ast}$ on $F$ and $z^{\ast}\left(  u_{k}\right)
=0.$ Then%
\[
0=f_{k}\left(  z^{\ast}\right)  \geq f\left(  z^{\ast}\right)  \geq g\left(
z^{\ast}\right)  =g\left(  y^{\ast}\right)  \geq0.
\]
This shows that $g=0$ and consequently, $\left[  0,f\right]  \cap
\operatorname*{FVL}\left[  E\right]  =\left\{  0\right\}  $ as required.
Therefore, we conclude that the free vector lattice $\operatorname*{FVL}%
\left[  E\right]  $ is not order dense in $\operatorname*{FBL}\nolimits^{(p)}%
\left[  E\right]  ,$ which completes the proof.
\end{proof}

Consider two Banach spaces $E$ and $F$ and a bounded operator
$T:F\longrightarrow E.$ Then as noted above we can associate to $T,$ in a
natural way, a lattice homomorphism $\overline{T}:\operatorname*{FBL}%
\nolimits^{(p)}\left[  F\right]  \longrightarrow\operatorname*{FBL}%
\nolimits^{(p)}\left[  E\right]  $ by extending $T$ to $\operatorname*{FBL}%
\nolimits^{(p)}\left[  F\right]  $ and then compose with the natural embedding
$E\hookrightarrow\operatorname*{FBL}\nolimits^{(p)}\left[  E\right]  .$ One of
the main ideas studied in \cite{L-925} is the relationship between operators
$T$ and $\overline{T}.$ In \cite{L-925} the authors showed that if
$\mathcal{P}$ is one of the properties -- injectivity, surjectivity and
bijectivity -- then $T$ has property $\mathcal{P}$ iff $\overline{T}$ does.
Among these, only the proof of injectivity requires order denseness. For this
reason we provide a corrected proof for it.

\begin{proposition}
Let $E$ and $F$ be two Banach spaces, $T\in L\left(  F,E\right)  $ and
$\overline{T}:\operatorname*{FBL}\nolimits^{(p)}\left[  F\right]
\longrightarrow\operatorname*{FBL}\nolimits^{(p)}\left[  E\right]  $ be the
associated lattice homomorphism introduced above. Then $T$ is injective iff
$\overline{T}$ is injective.
\end{proposition}

\begin{proof}
We need only to prove the forward implication. Assume then that $T$ is
injective and let $f\in\operatorname*{FBL}\nolimits^{(p)}\left[  F\right]  $
such that $\overline{T}f=0.$ Then $f$ vanishes on the range $T^{\ast}\left(
E^{\ast}\right)  $ of $T^{\ast}.$ As $f$ is bounded weak* continuous $f$
vanishes on $\overline{T^{\ast}\left(  E^{\ast}\right)  }^{bw^{\ast}}.$ As $T$
is injective $T^{\ast}$ has a weak* dense range and it is sufficient to show
that $\overline{T^{\ast}\left(  E^{\ast}\right)  }^{bw^{\ast}}=\overline
{T^{\ast}\left(  E^{\ast}\right)  }^{w^{\ast}}.$ But this follows from
Banach-Dieudonn\'{e} Theorem (Theorem \ref{BD}).
\end{proof}

In \cite[Theorem 3.4]{L-925} it was shown if $F$ is closed subspace of $E$ and
$\iota:F\hookrightarrow E$ is the inclusion map then $\overline{\iota
}:\operatorname*{FBL}\nolimits^{(p)}\left[  F\right]  \longrightarrow
\operatorname*{FBL}\nolimits^{(p)}\left[  E\right]  $ is order continuous.
However, the proof relies on order denseness of $\operatorname*{FVL}\left[
F\right]  $ in $\operatorname*{FBL}\nolimits^{(p)}\left[  F\right]  $ which
holds only in the finite-dimensional. The converse problem was also raised
(see \cite[Remark 3.5]{L-925}). What we have been were able to prove that the
former result is true for $p=\infty;$ for other values of $p,$ the question
remains open. Furthermore, we give a partial answer concerning the converse
problem. Before presenting the proof of these results let us state an
interesting result, which is very close to \cite[Lemma 2.9]{L-925}.

\begin{theorem}
\label{A}Let $U$ be a bounded weak* open subset of $E^{\ast}$ and $0\neq
x_{0}^{\ast}\in U.$ Then there exists $f$ in $\operatorname*{FBL}%
\nolimits^{(\infty)}\left[  E\right]  ^{+}$ such that $f\left(  x_{0}^{\ast
}\right)  \neq0$ and $f=0$ outside of $cone\left(  U\right)  .$ Moreover if
$h\in H_{\infty}\left[  E\right]  ^{+}$ satisfies $h\left(  x^{\ast}\right)
\geq1$ for all $x^{\ast}\in U$ then $h\geq f.$
\end{theorem}

Observe that $f$ is not necessarily dominated by $1$ on $U.$

An immediate yet very useful consequence is the following:

\begin{corollary}
\label{B}Let $U$ a bounded weak* open subset of $E^{\ast}$ and $\varepsilon
>0.$ Then there exists a nonzero element $f$ in $\operatorname*{FBL}%
\nolimits^{(\infty)}\left[  E\right]  ^{+}$ such that for every $h\in
H_{\infty}\left[  E\right]  ^{+}$ with $h\left(  x^{\ast}\right)
\geq\varepsilon$ we have $h\geq f.$
\end{corollary}

\begin{proof}
By definition of the bounded weak* topology \cite[Definition 2.7.1]{b-2554}
there exists a sequence $\left(  x_{n}\right)  $ in $E$ such that
$x_{n}\overset{\left\Vert .\right\Vert }{\longrightarrow}0$ and%
\[
B\left(  x_{0}^{\ast},\left(  x_{n}\right)  \right)  =\left\{  x^{\ast}\in
E^{\ast}:\left\vert x^{\ast}\left(  x_{n}\right)  -x_{0}^{\ast}\left(
x_{n}\right)  \right\vert <1\forall n\right\}  \subseteq U.
\]
We may assume without loss of generality that
\[
1=\left\vert x_{0}^{\ast}\left(  x_{1}\right)  \right\vert =\sup
\limits_{n}\left\vert x_{0}^{\ast}\left(  x_{n}\right)  \right\vert
\]
(take $\lambda U$ instead of $U$ and permute the $x_{n}$) Let
\[
V_{n}=\left\{  x^{\ast}\in E^{\ast}:\left\vert \left(  x^{\ast}-x_{0}^{\ast
}\right)  \left(  x_{i}\right)  \right\vert <\dfrac{1}{2}\text{ for
}i=1,...,n\right\}  .
\]
Let $g_{n}=\bigvee\limits_{k=1}^{n}\left\vert \delta_{x_{k}}\right\vert $ and
$f_{n}=\left(  \dfrac{1}{2}g_{n}-\bigvee\limits_{k=1}^{n}\left\vert
\delta_{x_{k}}-x_{0}^{\ast}\left(  x_{k}\right)  g_{n}\right\vert \right)
^{+}.$

\textbf{Claim 1 }$f_{n}\left(  x_{0}^{\ast}\right)  =\dfrac{1}{2}$ and for
$x^{\ast}\notin co\left(  V_{n}\right)  $ we have $f_{n}\left(  x^{\ast
}\right)  =0.$ The first equality follows easily from the fact that
$g_{n}\left(  x_{0}^{\ast}\right)  =1.$ The second is also clear if
$g_{n}\left(  x^{\ast}\right)  =0.$ Assume then that $g_{n}\left(  x^{\ast
}\right)  >0.$ Then $y^{\ast}:=\dfrac{x^{\ast}}{g_{n}\left(  x^{\ast}\right)
}\notin V_{n}.$ Then $\left\vert y^{\ast}\left(  x_{k}\right)  -x_{0}^{\ast
}\left(  x_{k}\right)  \right\vert \geq\dfrac{1}{2}$ for some $k\leq n$ and as
$g_{n}\left(  y^{\ast}\right)  =1$ we get
\[
f_{n}\left(  y^{\ast}\right)  \leq\left(  \dfrac{1}{2}-\left\vert y^{\ast
}\left(  x_{k}\right)  -x_{0}^{\ast}\left(  x_{k}\right)  \right\vert \right)
^{+}\leq0,
\]
which gives $f_{n}\left(  x^{\ast}\right)  =0$ and proves the claim.

We will show that $f:=\bigwedge\limits_{n=1}^{\infty}f_{n}$ fulfills the
desired conditions. Indeed, clearly $f\left(  x_{0}^{\ast}\right)  =\dfrac
{1}{2}.$ Let now $x^{\ast}\notin cone\left(  U\right)  .$ If $g_{n}\left(
x^{\ast}\right)  =0$ for some $n$ then $f_{n}\left(  x^{\ast}\right)  =0$ and
then $f\left(  x^{\ast}\right)  =0.$ Otherwise $g_{n}\left(  x^{\ast}\right)
>0$ for each $n.$ We claim that there exists $n$ in $\mathbb{N}$ such that
$\dfrac{x^{\ast}}{g_{n}\left(  x^{\ast}\right)  }\notin V_{n}.$ As above we
get $f_{n}\left(  \dfrac{x^{\ast}}{g_{n}\left(  x^{\ast}\right)  }\right)  =0$
which implies that $f\left(  x^{\ast}\right)  =0.$ It remains only to prove
the claim. We argue by contradiction and we assume that $\dfrac{x^{\ast}%
}{g_{n}\left(  x^{\ast}\right)  }\in V_{n}$ for each $n.$ Fix an integer $p$
and take $n\geq p.$ Then by definition of $V_{n},$
\[
\left\vert \dfrac{x^{\ast}}{g_{n}\left(  x^{\ast}\right)  }\left(
x_{p}\right)  -x_{0}^{\ast}\left(  x_{p}\right)  \right\vert <\dfrac{1}{2}.
\]
Letting $n\longrightarrow\infty$ we get%
\[
\left\vert \dfrac{x^{\ast}}{g\left(  x^{\ast}\right)  }\left(  x_{p}\right)
-x_{0}^{\ast}\left(  x_{p}\right)  \right\vert \leq\dfrac{1}{2},
\]
As this happens for every $p$ we get $\dfrac{x^{\ast}}{g\left(  x^{\ast
}\right)  }\in B\left(  x_{0}^{\ast},\left(  x_{n}\right)  \right)  \subseteq
U.$ But this contradicts our assumption that $x^{\ast}\notin co\left(
U\right)  .$

In the last step of the proof we will show that $f$ is in $\operatorname*{FBL}%
\nolimits^{\left(  \infty\right)  }\left[  E\right]  .$ In other words, $f$ is
bounded weak$^{\ast}$ continuous. It is enough to show that $S_{n}%
:=\bigwedge\limits_{k=1}^{n}f_{k}$ is a Cauchy sequence.

For $p\geq n,$
\[
0\leq S_{n}-S_{p}\leq f_{n}-\bigwedge\limits_{n}^{p}f_{k}=\bigvee
\limits_{n}^{p}\left(  f_{n}-f_{k}\right)  \leq\bigvee\limits_{n+1}%
^{p}\left\vert f_{n}-f_{k}\right\vert .
\]
On the other hand%

\begin{align*}
\left\vert f_{k}-f_{n}\right\vert  &  \leq\dfrac{1}{2}\left\vert g_{n}%
-g_{k}\right\vert +\left\vert \bigvee\limits_{i=1}^{k}\left\vert \delta
_{x_{i}}-x_{0}^{\ast}\left(  x_{i}\right)  g_{k}\right\vert ^{+}%
-\bigvee\limits_{i=1}^{n}\left\vert \delta_{x_{i}}-x_{0}^{\ast}\left(
x_{i}\right)  g_{n}\right\vert \right\vert \\
&  \leq\dfrac{1}{2}\left\vert g_{n}-g_{k}\right\vert +A+B
\end{align*}
with $A=\left\vert \bigvee\limits_{i=1}^{k}\left\vert \delta_{x_{i}}%
-x_{0}^{\ast}\left(  x_{i}\right)  g_{k}\right\vert ^{+}-\bigvee
\limits_{i=1}^{n}\left\vert \delta_{x_{i}}-x_{0}^{\ast}\left(  x_{i}\right)
g_{k}\right\vert \right\vert $

By Birkhoff Inequality we get%
\begin{align*}
\left\Vert A\right\Vert  &  \leq\bigvee\limits_{i=n+1}^{k}\left\Vert
\delta_{x_{i}}-x_{0}^{\ast}\left(  x_{i}\right)  g_{k}\right\Vert \\
&  \leq\bigvee\limits_{i=n+1}^{k}\left(  1+\left\Vert x_{0}^{\ast}\right\Vert
\left\Vert g_{k}\right\Vert \right)  \left\Vert x_{i}\right\Vert \leq
K\bigvee\limits_{i=n+1}^{k}\left\Vert x_{i}\right\Vert .
\end{align*}
On the other hand $B=\left\vert \bigvee\limits_{i=1}^{n}\left\vert
\delta_{x_{i}}-x_{0}^{\ast}\left(  x_{i}\right)  g_{k}\right\vert ^{+}%
-\bigvee\limits_{i=1}^{n}\left\vert \delta_{x_{i}}-x_{0}^{\ast}\left(
x_{i}\right)  g_{n}\right\vert \right\vert $ Using the inequality $\left\vert
\bigvee\limits_{i=1}^{n}a_{i}-\bigvee\limits_{i=1}^{n}b_{i}\right\vert
\leq\bigvee\limits_{i=1}^{n}\left\vert a_{i}-b_{i}\right\vert $ we see that
\[
B\leq\bigvee\limits_{i=1}^{n}\left\vert x_{0}^{\ast}\left(  x_{i}\right)
\right\vert \left\vert g_{k}-g_{n}\right\vert \leq\left\vert g_{k}%
-g_{n}\right\vert ,
\]
and then%
\[
\left\Vert B\right\Vert \leq\left\Vert g_{k}-g_{n}\right\Vert \leq
\bigvee\limits_{i=n+1}^{k}\left\Vert x_{i}\right\Vert .
\]
Combining all these we get%
\[
\left\Vert S_{p}-S_{n}\right\Vert \leq\left(  \dfrac{3}{2}+K\right)
\bigvee\limits_{i=n+1}^{p}\left\Vert x_{i}\right\Vert \longrightarrow0\text{
as }n\longrightarrow\infty.
\]
This completes the proof of the first part.

For the second part assume that $h$ satisfies the conditions of the lemma and
let $x^{\ast}\in E^{\ast}.$ If $g_{n}\left(  x^{\ast}\right)  =0$ for some $n$
then $0=f\left(  x^{\ast}\right)  \leq h\left(  x^{\ast}\right)  .$ Otherwise
$g\left(  x^{\ast}\right)  >0$. If $\dfrac{x^{\ast}}{g_{n}\left(  x^{\ast
}\right)  }\notin V_{n}$ for some $n$ then $0=f\left(  x^{\ast}\right)  \leq
h\left(  x^{\ast}\right)  .$ If not $\dfrac{x^{\ast}}{g_{n}\left(  x^{\ast
}\right)  }\in V_{n}$ for all $n$ and we get as above $\dfrac{x^{\ast}%
}{g\left(  x^{\ast}\right)  }\in U$ and $f\left(  \dfrac{x^{\ast}}{g\left(
x^{\ast}\right)  }\right)  \leq\dfrac{1}{2}\leq h\left(  \dfrac{x^{\ast}%
}{g\left(  x^{\ast}\right)  }\right)  ,$ which gives again $f\left(  x^{\ast
}\right)  \leq h\left(  x^{\ast}\right)  $ and we are done.
\end{proof}

Let $E$ be a Banach space and $A$ a subset of $E.$ We denote

\begin{center}
$E_{A}=\left\{  f\in\operatorname*{FBL}\nolimits^{\left(  \infty\right)
}\left[  E\right]  :f\left(  a\right)  =0\text{ for all }a\in A\right\}  .$
\end{center}

As every element $f\in\operatorname*{FBL}\nolimits^{(p)}\left[  E\right]  $ is
continuous with respect to the bounded weak* topology we have%
\[
E_{A}=E_{cone\left(  A\right)  }=E_{\overline{cone\left(  A\right)
}^{bw^{\ast}}}.
\]

The next result outlines basic properties of $E_{A}$ that will be used later.

\begin{proposition}
\label{C}Let $E$ be a Banach space and let $A,B$ be two nonempty subsets of
$E^{\ast}.$ Then the following statements hold.

\begin{enumerate}
\item[(i)] $E_{A}\subseteq E_{B}$ if and only if $B\subseteq\overline
{cone\left(  A\right)  }^{bw^{\ast}};$

\item[(ii)] $E_{A}^{d}=E_{E^{\ast}\diagdown\overline{cone\left(  A\right)
}^{bw^{\ast}}};$

\item[(iii)] $E_{A}^{dd}=E_{int\left(  \overline{\left(  cone\left(  A\right)
\right)  }^{bw^{\ast}}\right)  }.$
\end{enumerate}
\end{proposition}

\begin{proof}
(i) If $B\subseteq\overline{cone\left(  A\right)  }^{bw^{\ast}}$ then
$E_{A}\subseteq E_{B}$ by the above remark For the converse assume that $B$ is
not contained in $\overline{cone\left(  A\right)  }^{bw^{\ast}}$ and pick a
nonzero element $x_{0}^{\ast}$ in $B\diagdown\overline{cone\left(  A\right)
}^{bw^{\ast}}.$ As $\overline{cone\left(  A\right)  }^{bw^{\ast}}$ is a cone
$x_{0}^{\ast}\notin cone\left(  B\right)  .$ Applying Lemma \ref{A} to
$x_{0}^{\ast}$ and $U=E^{\ast}\diagdown\overline{cone\left(  A\right)
}^{bw^{\ast}}$ which is open to get a function $f\in\operatorname*{FBL}%
\nolimits^{\left(  \infty\right)  }\left[  E\right]  $ satisfying $f\left(
x_{0}^{\ast}\right)  \neq0$ and $f$ vanishes outside $cone\left(  U\right)
=U.$ In particular $f\in E_{A}\diagdown E_{B}.$ This proves (i).

(ii) Let $g\in E_{E^{\ast}\diagdown\overline{cone\left(  A\right)  }%
^{bw^{\ast}}}$ and $f\in E_{A}=E_{\overline{cone\left(  A\right)  }^{bw^{\ast
}}}.$ Then $\left\vert f\right\vert \wedge\left\vert g\right\vert $ vanishes
on $A$ as it is dominated by $\left\vert f\right\vert $, and vanishes on
$E^{\ast}\diagdown\overline{cone\left(  A\right)  }^{bw^{\ast}}$ as it
dominated by $\left\vert g\right\vert .$ Hence $\left\vert f\right\vert
\wedge\left\vert g\right\vert =0.$ This proves the inclusion $E_{E^{\ast
}\diagdown\overline{cone\left(  A\right)  }^{bw^{\ast}}}\subseteq E_{A}^{d}.$
Conversely let $g\in E_{A}^{d}$ and let $x_{0}^{\ast}\in E^{\ast}%
\diagdown\overline{cone\left(  A\right)  }^{bw^{\ast}}=U.$ Take now $f$ as in
(i). Then $0=\left(  \left\vert g\right\vert \wedge\left\vert f\right\vert
\right)  \left(  x_{0}^{\ast}\right)  $ and we get $g\left(  x_{0}^{\ast
}\right)  =0.$ This proves the second inclusion.

(iii) By (i) and (ii) we have%
\begin{align*}
E_{A}^{dd}  &  =\left(  E_{E^{\ast}\diagdown\overline{\left(  cone\left(
A\right)  \right)  }^{bw^{\ast}}}\right)  ^{d}=E_{E^{\ast}\diagdown
\overline{E^{\ast}\diagdown\overline{\left(  cone\left(  A\right)  \right)
}^{bw^{\ast}}}^{bw^{\ast}}}\\
&  =E_{int\left(  \overline{\left(  cone\left(  A\right)  \right)  }%
^{bw^{\ast}}\right)  },.
\end{align*}
where in the last equality we use the fact $X\diagdown\overline{Y}=int\left(
X\diagdown Y\right)  $ where $Y$ is a subset of a topological space $X.$
\end{proof}

It was shown in \cite{L-925} that if $F$ is a closed subspace of $E$ then
$\operatorname*{FBL}\nolimits^{(p)}\left[  F\right]  $ can be viewed as a
regular subspace of $\operatorname*{FBL}\nolimits^{(p)}\left[  E\right]  .$
More precisely Theorem 3.4 in \cite{L-925} shows that $\overline{\iota
}:\operatorname*{FBL}\nolimits^{(p)}[F]\rightarrow\operatorname*{FBL}%
\nolimits^{(p)}[E]$ \textit{is order continuous, where }$\iota
:F\longrightarrow E$ is the canonical injection. As the proof given in
\cite{L-925} uses the order denseness of $\operatorname*{FVL}\left[  E\right]
$ in $\operatorname*{FBL}\nolimits^{(p)}\left[  E\right]  $ it is not valid.
We are only able to show that its validity for $p=\infty$. In fact we will
show a more general result in this case. Before giving the proof we need some
preparation concerning bounded weak* topology. In \cite[Theorem 2]{a-2643}
Dieudonn\'{e} showed that if $F$ is a closed subspace of a Banach space $E$
and $\iota:F\hookrightarrow E$ is the natural embedding then $\iota^{\ast
}:E^{\ast}\longrightarrow F^{\ast}$ is open for the bounded weak* topology.
This will be crucial in the next result.

Let $S:F\longrightarrow G$ be an isomorphism between two Banach spaces,
$x_{0}^{\ast}\in G^{\ast}$ and $\left(  x_{n}\right)  $ a sequence in $G$ that
converges to $0.$ Then it is easily checked that $S^{\ast}\left(  B\left(
x_{0}^{\ast},\left(  x_{n}\right)  \right)  \right)  =B\left(  S^{\ast}%
x_{0}^{\ast},\left(  S^{-1}x_{n}\right)  \right)  .$ In particular $S^{\ast}$
is bounded weak* to bounded weak* open. This allows us to present a slightly
more general result than Dieudonn\'{e}'s.

\begin{lemma}
\label{DI}Let $F$ and $E$ two Banach spaces and $T:F\longrightarrow E$ an
embedding operator. Then $T^{\ast}:E^{\ast}\longrightarrow F^{\ast}$ is open
for the bounded weak* topology.
\end{lemma}

\begin{proof}
To see this write $T=i\circ S$ where $S$ is $T$ viewed as a map from $F$ to
$T\left(  F\right)  $ and $i:T\left(  F\right)  \longrightarrow E$ is the
canonical injection. Then use the above remark and Dieudonn\'{e}'s result.
\end{proof}

\begin{theorem}
Let $E$ and $F$ be two Banach spaces and $T:F\longrightarrow E$ a linear
bounded operator.

\begin{enumerate}
\item If $T$ is injective with closed range then $\overline{T}%
:\operatorname*{FBL}\nolimits^{(\infty)}[F]\rightarrow\operatorname*{FBL}%
\nolimits^{(\infty)}[E]$ \textit{is order continuous.}

\item If $\overline{T}$ is order continuous then $T$ is injective.
\end{enumerate}
\end{theorem}

\begin{proof}
(i) Assume by way of contradiction that $\overline{T}$ is not order
continuous. So we can find a net $\left(  f_{\alpha}\right)  $ in
$\operatorname*{FBL}\nolimits^{(\infty)}[F]$ such that $f_{\alpha}\downarrow0$
and $\overline{T}f_{\alpha}=f_{\alpha}\circ T^{\ast}\geq g>0$ for some $g$ in
$\operatorname*{FBL}\nolimits^{(\infty)}[E].$ Then there exist a real
$\varepsilon>0$ and a nonempty bounded weak* open set $V$ in $E^{\ast}$ such
that $g\geq\varepsilon$ on $V.$ It follows from Lemma \ref{DI} that
$U=T^{\ast}\left(  V\right)  $ is bounded weak* open in $F^{\ast}$ and in
addition $f_{\alpha}\geq\varepsilon$ on $U.$ Apply Corollary \ref{B} to find
an element $f$ in $\operatorname*{FBL}\nolimits^{(\infty)}[F]$ such that $0<f$
and $f_{\alpha}\geq f$ for all $\alpha\in A.$ This contradiction ends the proof.

(ii) Assume that $\overline{T}$ is order continuous. Then $\ker\overline{T}$
is a band in $\operatorname*{FBL}\nolimits^{(\infty)}[F].$ Using Proposition
\ref{C}.(iii) we get
\[
E_{T^{\ast}\left(  E^{\ast}\right)  }=\left(  \ker\overline{T}\right)
=\left(  \ker\overline{T}\right)  ^{dd}=E_{T^{\ast}\left(  E^{\ast}\right)
}^{dd}=E_{int\left(  \overline{T^{\ast}\left(  E^{\ast}\right)  }^{bw\ast
}\right)  }.
\]
Now according to part (i) of the same proposition we obtain:
\[
T^{\ast}\left(  E^{\ast}\right)  \subseteq\overline{int\left(  \overline
{T^{\ast}\left(  E^{\ast}\right)  }^{bw\ast}\right)  }^{bw^{\ast}}.
\]
In particular the subspace $\overline{T^{\ast}\left(  E^{\ast}\right)
}^{w\ast}=\overline{T^{\ast}\left(  E^{\ast}\right)  }^{bw\ast}$ has a
nonempty interior. Thus $\overline{T^{\ast}\left(  E^{\ast}\right)  }^{w\ast
}=F^{\ast}$ and $T$ is then injective.
\end{proof}

\begin{corollary}
\textit{Let} $F$ \textit{be a closed subspace of} $E$, \textit{and let}
$\iota:$ $F\hookrightarrow E$ \textit{be the inclusion map. Then}
$\operatorname*{FBL}\nolimits^{(\infty)}[F]$ \textit{is a regular} vector
\textit{sublattice of} $\operatorname*{FBL}\nolimits^{(\infty)}[E]$.
\end{corollary}

The last result we deal with is a recent result proved by Laustsen and
Tradacete in \cite{L-968}. However their proof used order denseness of
$\operatorname*{FVL}\left[  E\right]  $ in $\operatorname*{FBL}\nolimits^{(p)}%
\left[  E\right]  $ which is no longer correct. So we provide here a new proof
based on maximal elements introduced in \cite{L-403} for $\operatorname*{FBL}%
\left[  E\right]  $ and studied in \cite{L-979} for general
$\operatorname*{FBL}\nolimits^{(p)}\left[  E\right]  $. The sublattices
$I_{w^{\ast}}\left[  E\right]  $ and $H_{w^{\ast}}^{p}\left[  E\right]  $ of
$H\left[  E\right]  $ were recalled in the preliminary section.

\begin{proposition}
(see \cite[Propositoin 4.5]{L-968}) Let $E=\ell_{1}\left(  A\right)  $ for
some non-empty set $A$ and $p\in\lbrack1,\infty).$ Then the following equality
holds: $\overline{I_{w^{\ast}}\left[  \ell_{1}\left(  A\right)  \right]
}=H_{w^{\ast}}^{p}\left[  \ell_{1}\left(  A\right)  \right]  .$
\end{proposition}

\begin{proof}
The inclusion $I_{w^{\ast}}\left[  E\right]  \subseteq H_{w^{\ast}}^{p}\left[
E\right]  $ was proved in \cite{L-968} and is easy. Conversely let $0\leq f\in
H_{w^{\ast}}^{p}\left[  E\right]  .$ We know by \cite[Lemma 4.4]{L-403} and
\cite[Lemma 4.3]{L-979} that there exists a maximal element $g\in H^{p}$ such
that $g\geq f$. According to \cite[Lemmas 4.7 and 4.9]{L-403} and \cite[Lemmas
4.6 and 4.8]{L-979} we can write $g=g_{\varphi}$ for some $\varphi$ in
$\ell_{1}\left(  A\right)  .$ Recall that $g_{\varphi}\left(  x^{\ast}\right)
=\left\vert \varphi\left(  \left\vert x^{\ast}\right\vert ^{p}\right)
\right\vert ^{1/p}$ for all $x^{\ast}\in\ell_{\infty}\left(  A\right)
=E^{\ast}.$ Using \cite[Lemma 4.8]{L-403} and \cite[Lemma 4.7]{L-979} we know
that $g_{\varphi}\in\operatorname*{FBL}\nolimits^{(p)}\left[  E\right]
\subseteq\overline{I_{w}\left[  E\right]  }.$ Thus $f\in\overline{I_{w}\left[
E\right]  }$ as it was claimed.
\end{proof}


\begin{thebibliography}{99}                                                                                               %
\bibitem {b-240}C.D. Aliprantis and O. Burkinshaw, Positive Operators,
Springer, 2006.

\bibitem {b-1901}C. D. Aliprantis, R. Tourky, Cones and duality, $2007$

\bibitem {L-403}A. Avil\'{e}s, J. Rodr\'{\i}guez, and P. Tradacete, The free
Banach lattice generated by a Banach space. J. Funct. Anal. 274, No. 10,
2955--2977 (2018).

\bibitem {L-575}A. Avil\'{e}s, P. Tradacete, and I. Villanueva, The free
Banach lattices generated by $\ell_{p}$ and $c_{0}$. Rev. Mat. Complut. 32,
No. 2, 353--364 (2019).

\bibitem {L-979}Y. Azouzi, A. Ben Rjeb, P. Tradacete, The strong Nakano
property in Banach lattices. Rev. Real Acad. Cienc. Exactas Fis. Nat. Ser.
A-Mat. 119, 99 (2025). https://doi.org/10.1007/s13398-025-01764-7

\bibitem {L-816}E. Bilokopytov, Order continuity and regularity on vector
lattices and on lattices of continuous functions. Positivity 27, 52 (2023). https://doi.org/10.1007/s11117-023-01002-7

\bibitem {L-332}B. de Pagter, A.W. Wickstead, Free and projective Banach
lattices, Proc. Roy. Soc. Edinburgh Sect. A 145 (2015), no. 1, 105-143

\bibitem {a-2643}J. Dieudonn\'{e}, "Natural homomorphisms in Banach spaces,"
Proceedings of the American Mathematical Society, vol. 12, no. 6, 1961, pp. 888--894.

\bibitem {b-1417}Fabian, P. Habala, P. H\'{a}jek, V. Montesinos, V. Zizler,
Banach Space Theory. The Basis for Linear and Nonlinear Analysis, CMS Books in
Mathematics/Ouvrages de Math\'{e}matiques de la SMC, Springer, New York, 2011,
MR 2766381.

\bibitem {L-742}H. Jard\'{o}n-S\'{a}nchez, N.J. Laustsen, M.A. Taylor, P.
Tradacete, and V.G. Troitsky, Free Banach lattices under convexity conditions.
Rev. R. Acad. Cienc. Exactas F%
\'{}%
\i s. Nat. Ser. A Mat. RACSAM 116 (2022), no. 1, Paper No. 15.

\bibitem {L-968}N. J. Laustsen and P. Tradacete, Banach lattices of
homogeneous functions associated to a Banach space. Preprint, \textbf{arXiv}

\bibitem {b-2554}R.E. Megginson, An Introduction to Banach space Theory,
Graduate Texts in Mathematics 183, Springer-Verlag, New York, 1998.

\bibitem {L-925}T. Oikhberg, M. A. Taylor, P. Tradacete, V.G. Troitsky, Free
Banach lattices. J. Eur. Math. Soc. (2024), DOI 10.4171/JEMS/1552.
\end{thebibliography}
\end{document}